\documentclass[11pt]{article}
\usepackage{amssymb}
\usepackage[left=1in,top=1in,right=1in,bottom=1in]{geometry}
\usepackage{setspace}

\usepackage{amssymb, amsmath, amsthm, graphicx,mathrsfs}

\def\qed{\hfill\ifhmode\unskip\nobreak\fi\quad\ifmmode\Box\else\hfill$\Box$\fi}
\def\ite#1{\hfill\break${}$\hbox to 50pt {\quad(#1)\hfill}}
\newtheorem{thm}{Theorem}[section]
\newtheorem{cor}[thm]{Corollary}

\newtheorem{rem}[thm]{Remark}
\newtheorem{lem}[thm]{Lemma}

\newtheorem{claim}[thm]{Claim}

\def\cA{{\mathcal A}}
\def\cB{{\mathcal B}}

\def\cD{{\mathcal D}}
\def\cE{{\mathcal E}}
\def\cF{{\mathcal F}}
\def\cG{{\mathcal G}}
\def\cH{{\mathcal H}}
\def\cL{{\mathcal L}}
\def\cM{{\mathcal M}}
\def\cK{{\mathcal K}}
\def\cP{{\mathcal P}}

\begin{document}

\title{\vspace{-0.5in}
  Berge cycles in non-uniform hypergraphs }

\author{
{{Zolt\'an F\"uredi}}\thanks{
\footnotesize {Alfr\'ed R\'enyi Institute of Mathematics, Hungary.
E-mail:  \texttt{z-furedi@illinois.edu}. 
Research supported in part by the Hungarian National Research, Development and Innovation Office NKFIH grant KH-130371. 
}}
\and{{Alexandr Kostochka}}\thanks{
\footnotesize {University of Illinois at Urbana--Champaign, Urbana, IL 61801
 and Sobolev Institute of Mathematics, Novosibirsk 630090, Russia. E-mail: \texttt {kostochk@math.uiuc.edu}.
 Research is supported in part by NSF grant  DMS-1600592
and grants 18-01-00353A and 19-01-00682  of the Russian Foundation for Basic Research.
}}
\and{{Ruth Luo}}\thanks{University of California, San Diego, La Jolla, CA 92093, USA. E-mail: {\tt ruluo@ucsd.edu}. Research of this author
is supported in part by NSF grant DMS-1902808.
}}
\date{\today}

\maketitle

\vspace{-0.3in}

\begin{abstract}

We consider two extremal problems for set systems without long Berge cycles. First we give Dirac-type minimum degree conditions that force long Berge cycles. Next we give an upper bound for the number of hyperedges in a hypergraph with bounded circumference. Both results are best possible in infinitely many cases. 

\end{abstract}

\section{Introduction}

\subsection{Classical results on longest cycles in graphs}  

The {\em circumference} $c(G)$ of a graph $G$ is the length of its longest cycle.
In particular, if a graph has a cycle $C$ which covers all of its vertices, $V(C)=V(G)$, we say it is {\em hamiltonian}.
A classical result of Dirac states that high minimum degree in a graph forces hamiltonicity.

\begin{thm}[Dirac~\cite{D}]\label{th:D} Let $n \geq 3$, and let $G$ be an $n$-vertex graph with minimum degree $\delta(G)$. 
If $\delta(G) \geq n/2$, then $G$ contains a hamiltonian cycle.  
If $G$ is 2-connected, then  $c(G)\geq \min\{n, 2\delta(G)\}$. 
\end{thm}

Inspired by this theorem, it is common in extremal combinatorics to refer to results in which a minimum degree condition forces some structure as a {\em Dirac-type condition}.
The second part of Theorem~\ref{th:D} cannot be extended to non 2-connected graphs: 
 let $\cF_{n,k}$ be the family of graphs in which each block (inclusion  maximal 2-connected subgraph) of the graph is a copy of $K_{k-1}$. 
Every $F \in \cF_{n,k}$ has minimum degree $k-2$, but its longest cycle has length $k-1$.

\begin{thm}[Erd\H{o}s, Gallai~\cite{EG}]\label{th:EG} Let $G$ be an $n$-vertex graph with no cycle of length $k$ or longer. Then $e(G) \leq \frac{n-1}{k-2}{k-1 \choose 2}$.
\end{thm}

So the graphs in $\cF_{n,k}$ have the maximum number of edges among the  $n$-vertex graphs with circumference $k-1$. 
They also maximize the number of cliques of any size:

\begin{thm}[Luo~\cite{luo}]\label{cliques} Let $G$ be an $n$-vertex graph with no cycle of length $k$ or longer. Then the number of copies of $K_r$ in $G$ is at most $\frac{n-1}{k-2}{k-1 \choose r}$.
\end{thm}

\subsection{Known results on cycles in hypergraphs}

A hypergraph $\cH$ is a set system.  We often refer to the ground set as the set of  vertices $V(\cH)$ of $\cH$ and to the sets as the hyperedges $E(\cH)$ of $\cH$. When there is no ambiguity, we may also refer to the hyperedges as edges. In this paper, we prove versions of Theorems~\ref{th:D} and~\ref{th:EG} for hypergraphs with no restriction on edge sizes. Namely, we seek long {\em Berge cycles}.

A {\bf Berge cycle} of length $\ell$ in a hypergraph is a set of $\ell$ distinct vertices $\{v_1, \ldots, v_\ell\}$ and $\ell$ distinct edges $\{e_1, \ldots, e_\ell\}$ such that $\{ v_{i}, v_{i+1} \}\subseteq   e_i$ with indices taken modulo $\ell$. The vertices $\{v_1, \ldots, v_\ell\}$ are called {\bf representative vertices} of the Berge cycle.

  A {\bf Berge path} of length $\ell$ in a hypergraph  is a set of $\ell+1$ distinct vertices $\{v_1, \ldots, v_{\ell+1}\}$ and $\ell$ distinct hyperedges $\{e_1, \ldots, e_{\ell}\}$ such that $\{ v_{i}, v_{i+1} \}\subseteq   e_i$ for all $1\leq i\leq \ell$. The vertices $\{v_1, \ldots, v_{\ell+1}\}$ are called {\bf representative vertices} of the Berge path.

For a hypergraph $\cH$, the {\bf 2-shadow} of $\cH$, denoted $\partial_2 \cH$, is the graph on the same vertex set such that $xy \in E(\partial_2\cH)$ if and only if $\{x,y\}$ is contained in an edge of $\cH$.

Note that if we require no conditions on multiplicities of hyperedges, then we can arbitrarily add hyperedges of size $1$ without creating new Berge cyles or Berge paths. From now on, we only consider {\em simple} hypergraphs, i.e., those without multiple edges (except if it is stated otherwise).

Bermond, Germa, Heydemann, and Sotteau~\cite{bermond} were among the first to prove Dirac-type results for uniform hypergraphs without long Berge cycles: Let $k > r$ and $\cH$ be an $r$-uniform hypergraph with minimum degree $\delta(\cH)  \geq {k-2 \choose r-1} + (r-1)$, then $\cH$ contains a Berge cycle of length at least $k$.
For large $n$, generalizations and results for linear hypergraphs are proved by   Jiang and  Ma~\cite{JM}.
 Ma,  Hou, and  Gao~\cite{MHG_r3} studied $3$-uniform hypergraphs.
Coulson and Perarnau~\cite{rainbow} proved that if $\cH$ is an $r$-uniform hypergraph on $n$ vertices, $r = o(\sqrt{n})$, and $\cH$ has minimum degree $\delta(\cH) > {\lfloor (n-1)/2 \rfloor \choose r-1}$, then $\cH$ contains a Berge hamiltonian cycle.

Our new results differ from these in several aspects. We consider non-uniform hypergraphs, prove exact formulas, prove results for every $n$ (or every $n> 14$), and use only classical tools mentioned above and in  Section~\ref{ss0}.

\section{New results} 
Our first result is a Dirac-type condition that forces hamiltonian Berge cycles.

\begin{thm}\label{diracH}
Let $n \geq 15$  
and let $\cH$ be an $n$-vertex hypergraph such that $\delta(\cH) \geq 2^{(n-1)/2} + 1$ if $n$ is odd, or $\delta(\cH) \geq 2^{n/2 -1} + 2$ if $n$ is even. Then $\cH$ contains a Berge hamiltonian cycle.
\end{thm}

The following four constructions show that Theorem~\ref{diracH} is the best possible.
\newline
${}$\enskip --- \enskip Let $n$ be odd. Let $\cH$ be the $n$-vertex hypergraph on the ground set $[n]$ with edges $\{A: A \subseteq [(n+1)/2]\} \cup \{B: B \subseteq \{(n+1)/2, \ldots n\}\}$. Then $\delta(\cH) = 2^{(n-1)/2}$ and $\cH$ has no hamiltonian Berge cycle (because it has a cut vertex).
\newline
${}$\enskip --- \enskip Let $n$ be even. Let $\cH$ be the $n$-vertex hypergraph on the ground set $[n]$ with edges $\{A: A \subseteq [n/2]\} \cup \{B: B \subseteq \{(n/2 +1, \ldots n\}\}$ and the set $[n]$. Then $\delta(\cH) = 2^{n/2 -1}+1$ and $\cH$ has no hamiltonian Berge cycle (because it has a cut edge, $[n]$).
\newline
${}$\enskip --- \enskip Let $n$ be odd.  Let $\cH$ be the $n$-vertex hypergraph on the ground set $[n]$ obtained by taking all hyperedges with at most one vertex in $[(n+1)/2]$. Then $\delta(\cH) = 2^{(n-1)/2}$, and $\cH$ cannot contain a Berge cycle with two consecutive representative vertices in $[(n+1)/2]$. 
\newline
${}$\enskip --- \enskip Let $n$ be even. Let $\cH$ be the $n$-vertex hypergraph on the ground set $[n]$ obtained by taking all hyperedges with at most one vertex in $[n/2 + 1]$ and the edge $[n]$. Then $\delta(\cH) = 2^{n/2 - 1} + 1$, and $\cH$ cannot contain a Berge cycle with two instances of two consecutive representative vertices in $[n/2 + 1]$ (because only one edge of $\cH$ contains multiple vertices in $[n/2+1]$).

Next, we consider hypergraphs without long Berge paths or cycles.

\begin{thm}\label{th:mindeg_path}
Let $k \geq 2$ and let $\cH$ be a hypergraph such that $\delta(\cH) \geq 2^{k-2}+1$.
Then $\cH$ contains a Berge path with $k$ base vertices.
\end{thm}

A vertex disjoint union of complete hypergraphs of $k-1$ vertices shows that this bound is  best possible for $n:=|V(\cH)|$  divisible by $(k-1)$.
It would be interesting to find  $\max \delta(\cH)$ for other values of $n$, and also for the cases when $\cH$ is
  connected or 2-connected.

\begin{thm}\label{th:mindeg_cycle}
Let $k \geq 3$ and let $\cH$ be a hypergraph such that $\delta(\cH) \geq 2^{k-2}+2$.
Then $\cH$ contains a Berge cycle of length at least $k$.
\end{thm}

The following constructions show that the bound in Theorem~\ref{th:mindeg_cycle} is  best possible when $n$ is divisible by $(k-1)$ and also when $n\equiv 1 \mod (k-1)$ for $n>(k-1)(2^{k-2}+1)$. 
In the first case, take a vertex disjoint union of complete hypergraphs with $k-1$ vertices and add one more set, namely $[n]$. 
In the other case, take $m:=(n-1)/(k-1) \geq 2^{k-2}+1$ disjoint $(k-1)$-sets $A_1, \dots, A_m$ and an element $x$ such that $[n]=(\cup _{1\leq i\leq m}A_i)\cup  \{ x\} $. Then define $\cH$ as the union of complete hypergraphs on
 the sets $A_i$'s together with the hyperedges of the form $A_i \cup \{ x\} $. 
If we do not insist on connectedness, then $(2^{k-2}+1)$-regular examples can be constructed for \emph{all} $n\geq k^2 2^{k-2}$.

Finally, we prove a hypergraph version of Theorem~\ref{th:EG}. 

\begin{thm}\label{th:EGh}Let $n \geq k \geq 3$ and let $\cH$ be an $n$-vertex hypergraph with no Berge cycle of length $k$ or longer. Then \[e(\cH) \leq
 2+ \frac{n-1}{k-2}\left(2^{k-1}-2\right) 
.\]
\end{thm}
The bound in Theorem~\ref{th:EGh} is best possible when $n \equiv 1 \mod (k-2)$. Take $m:= (n-1)/(k-2)$ and disjoint sets $A_1 , \ldots, A_m$ of size $k-2$. Let $x$ be a new element, and set $[n] = (\cup _{1\leq i\leq m}A_i)\cup  \{ x\}$. Define $\cH$ to be the union of all sets $A$ such that there exists an $i$ with $A \setminus \{x\} \subseteq A_i$. Note that the 2-shadow $\partial_2(\cH)$ is in the family $\cF_{n,k}$ defined before Theorem~\ref{th:EG}.

There are many exact results concerning the maximum size of uniform hypergraphs avoiding Berge paths and cycles, see the recent results of Ergemlidze et al.~\cite{EGyMSTZ}
or one by the present authors~\cite{FKL}.

\section{Dirac type conditions for hamiltonian hypergraphs}
In this section, we present a proof for Theorem~\ref{diracH}. The proof method relies on reducing the hypergraph to a dense nonhamiltonian graph.
In the next three subsections we collect some results  about such graphs.
Subsections~\ref{ss33} and~\ref{ss34} contain the proof for hypergraphs.

\subsection{Classical tools}\label{ss0}

Let $G$ be an $n$-vertex graph. The {\em hamilton-closure} of $G$ is the unique  graph $C(G)$ of order $n$  that can be obtained from $G$ by recursively joining nonadjacent vertices with degree-sum at least $n$.

\begin{thm}[Bondy, Chv\'atal~\cite{BC}]\label{BCthm}
If $C(G)$ is hamiltonian, then so is $G$.
\end{thm}

A graph $G$ is called {\em hamiltonian-connected} if for any pair of vertices $x,y \in V(G)$ there is a hamiltonian $(x,y)$-path.
The following corollary can be  obtained from Theorem~\ref{BCthm} or from the classical result of P\'osa~\cite{Posa}:
If for every pair of nonadjacent vertices $x,y \in V(G)$ we have $d(x) + d(y) \geq |V(G)| + 1$, then $G$ is hamiltonian-connected.

\begin{cor}\label{cor1}
If $e(G)\geq \binom{n}{2}-2$ and $n\geq 5$ then $G$ is hamiltonian-connected.  \qed
\end{cor}

We will need the following result about the structure of matchings in bipartite graphs. It is a well known fact in the theory of transversal matroids (but one can also give a short, direct proof
 finding an $M_3\subseteq M_1\cup M_2$). 

\begin{thm}\label{matroid}
Let $G[X,Y]$ be a bipartite graph. Suppose that there is a matching $M_1$ in $G$ joining the vertices of $X_1\subseteq X$ and $Y_1\subseteq Y$. 
Suppose also that we have another matching $M_2$ with end vertices $X_2\subseteq X$ and $Y_2 \subseteq Y$ such that $Y_2\subseteq Y_1$. 
Then there exists a third matching $M_3$ from $X_3 \subseteq X$ to $Y_3 \subseteq Y$ such that 
\begin{equation*}
  Y_3=Y_1 \quad \text{and} \quad X_3\supseteq X_2. 
\end{equation*}
\end{thm}

\begin{thm}[Erd\H{o}s~\cite{Erdos}]\label{Erdos} Let $n, d$ be integers with $1 \leq d \leq \left \lfloor \frac{n-1}{2} \right \rfloor$, and set $h(n,d):={n-d \choose 2} + d^2$.
If $G$ is a nonhamiltonian graph on $n$ vertices with minimum degree $\delta(G) \geq d$, then
     \[e(G) \leq \max\left\{ h(n,d),h(n, \left \lfloor \frac{n-1}{2} \right \rfloor)\right\}=:e(n,d).\] 
\end{thm}

\subsection{A lemma for nonhamiltonian graphs}\label{ss11}

The lemma below follows from a result of Voss~\cite{Voss} (and from the even more detailed descriptions by Jung~\cite{Jung} and Jung, Nara~\cite{JungNara}). We only state and use a weaker version and for completeness include a short proof. Define five classes of nonhamiltonian graphs.

${}$\enskip --- \enskip Let $n=2k+2$, $V=V_1\cup V_2$, $|V_1|=|V_2|=k+1$, ($V_1\cap V_2=\emptyset$). We say that $G\in \cG_1$ 
 if its edge set is the union of two complete graphs with vertex sets $V_1$ and $V_2$ and it contains at most one further edge $e_0$ (joining $V_1$ and $V_2$);
  \newline
 ${}$\enskip --- \enskip Let $n=2k+1$, $V=V_1\cup V_2$, $|V_1|=|V_2|=k+1$, $V_1\cap V_2=\{ x_0\}$. We say that $G\in \cG_2$
 if its edge set is the union of two complete graphs with vertex sets $V_1$ and $V_2$; 
\newline
${}$\enskip --- \enskip Let $n=2k+2$, $V=V_1\cup V_2$, $|V_1|=k+1$, $|V_2|=k+2$, $V_1\cap V_2=\{ x_0\}$. We say that $G\in \cG_3$
 if its edge set is the union of a complete graph with vertex set $V_1$ 
  and a $2$-connected graph $G_2$ with vertex set $V_2$ such that $\deg_G(v)\geq k$ for every vertex $v\in V$;
  \newline
${}$\enskip --- \enskip Let $n=2k+1$, $V=V_1\cup V_2$, $|V_1|=k$, $|V_2|=k+1$, ($V_1\cap V_2=\emptyset$).  We say that $G\in \cG_4$ 
 if $V_2$ is an independent set, and its edge set contains all edges joining $V_1$ and $V_2$;
\newline
${}$\enskip --- \enskip Let $n=2k+2$, $V=V_1\cup V_2$, $|V_1|=k$, $|V_2|=k+2$, ($V_1\cap V_2=\emptyset$).  We say that $G\in \cG_5$ if 
 $V_2$ contains at most one edge $e_0$ and $\deg_G(v)\geq k$ for every vertex $v\in V$ (so its edge set contains all but at most two edges joining $V_1$ and $V_2$).

\begin{lem}\label{5G}Let $k\geq 3$ be an integer, $n \in \{ 2k+1, 2k+2\} $.
Suppose that $G$ is an $n$-vertex nonhamiltonian graph with $\delta(G) \geq k = \lfloor(n-1)/2\rfloor$, $V:= V(G)$.
Then $G\in \cG_1\cup \dots \cup \cG_5$.  
\end{lem}

\begin{proof}
Suppose first that $G$ is not 2-connected.
Then there exist two blocks $B_1, B_2$ of $G$ (i.e., $B_i$ is a maximal 2-connected subgraph or a $K_2$) which are {\em endblocks}, i.e., for $i=1,2$ there is a vertex $v_i\in B_i$ such that $V(B_i)\setminus \{ v_i\}$ does not meet any other block.
Then $\{v\} \cup N(v) \subset V(B_i)$ for all $v\in V(B_i)\setminus \{ v_i\}$, so
an endblock has at least $k+1$ vertices and if $|V(B_i)|=k+1$ then it is a clique.
If $B_1$ and $B_2$ are disjoint then we get $n=2k+2$, and $G\in \cG_1$.
If $B_1$ and $B_2$ meet, then $G$ has no other blocks, and  $G\in \cG_2\cup \cG_3$.

Suppose now that $G$ is 2-connected. By the second part of Dirac's theorem (Theorem~\ref{th:D}), the length of a longest cycle $C$ of $G$ is at least $2k$. If $|V(C)| = n-1$,  assume $C=v_1 \ldots v_{n-1} v_1$ and $v_n \notin V(C)$.
Then $v_n$ has at least $k$ neighbors in $C$, with no two of them appearing consecutively (otherwise we could extend $C$ to a hamiltonian cycle). Without loss of generality, let $N(v_n) = \{v_1, v_3, \ldots, v_{2k-1}\}$.
If for some $i<j$ such that $v_i, v_j \in N(v_n)$, $v_{i+1}v_{j+1} \in E(G)$, then we obtain the hamiltonian cycle $v_1 v_2 \ldots v_i v_n v_j v_{j-1} \ldots v_{i+1}v_{j+1} v_{j+2} \ldots v_{n-1}v_1$. Therefore the vertices in $C$ of even parity, together with $v_n$, form an independent set. In case of $n=2k+1$ we got $G\in \cG_4$.
If $n=2k+2$ then in the same way we get that $\{v_{2k+1}\}\cup \{ v_2, v_4, \dots, v_{2k-2} \}$ together with $v_n$ is also independent, so the set $\{ v_2, ..., v_{2k-2}\}
 \cup \{ v_{2k}, v_{2k+1}, v_n \}$ contains only the edge $v_{2k}v_{2k+1}$, $G\in \cG_5$.

Finally, consider the case  that $|V(C)| = n-2$, (i.e., $n=2k+2$) and let $x, y \notin V(C)$.
We claim that $xy\notin E(G)$. Indeed, suppose to  the contrary, that  $xy \in E(G)$. Without loss of generality,
 $A:=\{v_1, v_3, \ldots, v_{2k-3}\} \subseteq N(x)$ or $(A\setminus  \{ v_{2k-3}\}) \cup \{ v_{2k-2}\})\subseteq N(x)$. Note that for any $v_i \in N(x)$, $\{v_{i-2}, v_{i-1}, v_{i+1}, v_{i+2}\} \cap N(y) = \emptyset$ (indices are taken modulo $2k$), because we can remove a segment of $C$ with at most 3 vertices and replace it with a segment with at least 4 containing the edge $xy$. 
 This leads to a contradiction because there is not enough room
    on the $2k$-cycle $C$ to distribute the at least $k-1$ vertices of $N(y)-x$. 
    
If $xy \notin E(G)$ then without loss of generality  $N(x) = \{v_1, v_3,\ldots v_{2k-1}\}$. 
Then the set $\{x\} \cup \{v_2, \ldots, v_{2k}\}$ is an independent set. 
If $y v_i \in E(G)$ for some $i \in \{2, 4, \ldots, {2k}\}$, then because $y$ has $k$ neighbors in $C$ and no two of them appear consecutively, $N(y) = \{v_2, v_4, \ldots, v_{2k}\}$, and we obtain a hamiltonian cycle by replacing the segment $v_1v_2v_3v_4$ of $C$ with the path $v_1 x v_3 v_2 y v_4$. 
Therefore $V_2:=\{v_2, v_4, \ldots, v_{2k}\} \cup \{x, y\}$ is an independent set of size $k+2$, and so  $G\in \cG_5$.
\end{proof}

\subsection{A maximality property of the  graphs in   $\cG_1\cup\ldots\cup\cG_5$}

Let $G\in \cG_1\cup \dots \cup \cG_5$ be a graph.
Delete a set of edges $\cA$ from $E(G)$ where $|\cA|\leq 1$ for $G\in \cG_2\cup \cG_3\cup \cG_4$
and  $|\cA|\leq 2$ for $G\in \cG_1\cup \cG_5$.
Then add a set of new edges $\cB$ as defined below:
\newline
${}$\enskip --- \enskip
For $G\in \cG_1$, $|\cB|=2$ and it consists of any two disjoint pairs joining $V_1$ and $V_2$;
\newline
${}$\enskip --- \enskip
for $G\in \cG_2\cup\cG_3$, $|\cB|=1$ and it consists of any pair $x_1x_2$ joining $V_1\setminus \{ x_0\}$ and $V_2\setminus \{ x_0\}$ (here $x_i\in V_i$);
\newline
${}$\enskip --- \enskip
for $G\in \cG_4$, $|\cB|=1$ and it consists of any pair contained in $V_2$;
\newline
${}$\enskip --- \enskip
and for $G\in \cG_5$, $|\cB|=2$ and it consists of any two distinct pairs contained in $V_2$.

\begin{lem}\label{5Gmodified} If $k\geq 6$,
then the graph $\left(E(G)\setminus \cA\right) \cup \cB$ defined by the above process is hamiltonian, except if $G\in \cG_3$,
$x_0$ has exactly two neighbors $x_2$ and $y_2$ in $V_2$, $\cA= \{ x_0y_2\}$, $\cB= \{ x_1x_2\}$, and $G[V_2\setminus \{ x_0\}]$ is either a $K_{k+1}$ or misses only the edge $x_2y_2$.
\end{lem}

\begin{proof}
If $G\in \cG_1$ and we add two disjoint edges $x_1x_2$ and $y_1y_2$ joining $V_1$ and $V_2$ ($x_1, y_1\in V_1$) then to form a hamiltonian cycle
we need an $x_1x_2$ path $P_1$, and a $y_1y_2$ path $P_2$ of length $k$, $V(P_i)=V_i$ and $E(P_i)\subset E(G)\setminus \cA$. Such paths exist because the graph  $G[V_i]\setminus \cA$ has at least $\binom{k+1}{2}-2$ edges, so it satisfies the condition of Corollary~\ref{cor1}.

If $G\in \cG_2\cup \cG_3$ and we add an edge $x_1x_2$ joining $V_1\setminus\{x_0\}$ and $V_2\setminus\{x_0\}$  then
we need paths $P_1$, $P_2$ of length $|V_i|-1$ joining $x_i$ to $x_0$, $V(P_i)=V_i$ and $E(P_i)\subset E(G)\setminus \cA$.
If $G[V_i]\setminus \cA$ satisfies the condition of Corollary~\ref{cor1} then we can find $P_i$.
The only missing case is when $|V_2|=k+2$ (so $G\in \cG_3$).
Let $G_2$ be the graph on $|V_2|+1$ vertices obtained from $G[V_2]\setminus \cA$ by adding a new vertex $x_2'$ and two edges $x_0x_2'$ and $x_2x_2'$.
If $G_2$ has a hamiltonian cycle $C$ then it should contain  $x_0x_2'$ and $x_2x_2'$ so the rest of the edges of $C$ can serve as $P_2$ we are looking for. Consider the hamilton-closure $C(G_2)$ and
apply Theorem~\ref{BCthm} to $G_2$. Since the degrees of $V_2\setminus \{ x_0\}$ in $G_2$ are at least $k-1$ and $2(k-1)\geq k+3= |V(G_2)|$, $C(G_2)$ is a complete graph on $V_2\setminus \{ x_0\}$.
So $C(G_2)$ is hamiltonian unless the only neighbors of $x_0$ in $G_2$ are $x_2$ and $x_2'$. Hence $N_G(x_0)\cap V_2= \{ x_2, y_2\}$ and $\cA=\{ x_0y_2\}$.

The last case is when $G\in \cG_5$, $|\cA|=2$, $\cB= \{ e_1,e_2\}$ (two distinct edges inside $V_2$). (The proofs of the other cases, especially when  $G\in \cG_4$ are easier).
We create a graph $H_0$ from $G$ as follows:
Delete the edge $e_0$ (if it exists), delete the edges of $\cA$ joining $V_1$ and $V_2$,
add two new vertices $z_1, z_2$  to $V_1$ and join $z_i$ to the endpoints of $e_i$.
We obtain the graph $H$ by adding all possible $\binom{k+2}{2}$ pairs from $V_1\cup \{ z_1,z_2\}$ to $H_0$.

If $H$ is hamiltonian then its hamiltonian cycle must use only edges of $H_0$ (because $V_2$ is an independent set of size $k+2$ in $H$).
If the 
 graph $H_0$ is hamiltonian then its hamiltonian cycle must use the two edges of the degree $2$ vertex $z_i$, so $\left(G\setminus (\{ e_0\}\cup \cA)\right) \cup \cB$ is hamiltonian as well. 
So it is sufficient to show that $H$ has a hamiltonian cycle. 

Let $A$ be the graph on $V(H)$ consisting of the edges of $\cA$ joining $V_1$ and $V_2$ together with the (at most) two missing pairs $E(K(V_1, V_2))\setminus E(G)$.    
We will again apply  Theorem~\ref{BCthm} to $H$, so consider the hamilton-closure $C(H)$.
The degree $\deg_H(x)$ of an $x\in V_1$ is $(2k+3)-\deg_A(x)$.
The degree $\deg_H(y)$ of a $y\in V_2$ is at least $|V_1|-\deg_A(y)= k-\deg_A(y)$.
Since $\deg_A(x)+\deg_A(y)\leq |E(A)|+1\leq 5$ we get for $k\geq 6$ that
\[  \deg_H(x)+ \deg_H(y)\geq (3k+3)-\left(\deg_A(x)+\deg_A(y)\right)\geq 3k-2 \geq 2k+4= |V(H)|.
\]
So $C(H)$ contains the complete bipartite graph $K(V_1,V_2)=K_{k,k+2}$. 
Then it is really a simple task to find a hamiltonian cycle in $C(H)$ 
 and therefore $\left(E(G)\setminus \cA\right) \cup \cB$ is hamiltonian.  
\end{proof}

\subsection{Proof of Theorem~\ref{diracH}, reducing the hypergraph to a dense graph}\label{ss33}

 Fix $\cH$ to be an $n$ vertex hypergraph satisfying the minimum degree condition. 
We will find a hamiltonian Berge cycle in $\cH$.

Recall that $H=\partial_2(\cH)$ denote the 2-shadow of $\cH$, a graph on $V=V(\cH)$.
Define a bipartite graph $B:= B[E(\cH), E(H)]$ with parts  $E(\cH)$ and $E(H)$ and with edges $\{ h, xy\}$ where
 a hyperedge $h\in E(\cH)$ is joined to the graph edge $xy\in E(H)$ if $\{ x, y\} \subseteq h$.
In the case of $\{ x,y\} \in \cH$ we consider the edge $xy\in E(H)$ and $\{ x,y\} \in E(\cH)$ as two distinct objects of $B$, so $B$ is indeed a bipartite graph (with $|E(\cH)|+| E(H)|$ vertices and no loops).
Let $M$ be a maximum matching of $B$. So 
$M$ can be considered as a partial injection of maximum size, i.e., a bijection $\phi$ between two subsets $\cM\subseteq E(\cH)$ and $\cE\subseteq E(H)$
  such that $|\cM|=|\cE|$,  $\phi(m)\subseteq m$ for $m\in \cM$ (and $\phi(m_1)\neq \phi(m_2)$ for $m_1\neq m_2$).
Consider the subgraph $G=(V,\cE)$ of $H$.   Then $G$ does not have a hamiltonian cycle, otherwise by replacing the edges of a hamiltonian cycle with their corresponding matched hyperedges in $M$, we obtain a hamiltonian Berge cycle in $\cH$ (with representative vertices in the same order).
In this subsection we are going to prove that
\begin{equation}\label{eq331}
  \delta(G) \geq \lfloor (n-1)/2 \rfloor:=k.
\end{equation}
Since $G$ has no hamiltonian cycle and $k\geq 7$, if~\eqref{eq331} holds, then by Lemma~\ref{5G}, $G\in \cG_1\cup \dots\cup \cG_5$. We will consider this case and prove the remainder of Theorem~\ref{diracH} in the next subsection.

Let $\cH_2:= E(\cH)\cap \partial_2(\cH)$, the set of 2-element edges of $\cH$.
We may assume that among all maximum sized matchings of $B$ the matching $M$ maximizes $|\cM\cap \cH_2|$.
\begin{claim}\label{cl33}
  $\cH_2\subseteq \cM$, $\partial_2(\cM)= E(H)$, and every $m\in E(\cH)\setminus \cM$ induces a complete graph in $G$.
  \end{claim}

\begin{proof}
If $m\in E(\cH)$ contains an edge $e\in E(H)\setminus E(G)$ then one can enlarge the matching $M$ by adding $\{m,e\}$ to it, if it is possible.
Since $M$ is maximal, it cannot be enlarged, so $m\in \cM$. This implies the second and the third statements.
We also obtained that if $\{x,y\}\in E(\cH)$ then $xy\in E(G)$, so
  $\phi(m)=xy$ for some $m\in \cM$.  In case of $|m|>2$ we can replace the pair $\{ m, xy\}$ by the pair $\{ \{ x,y\}, xy\}$ in $M$ and the new matching covers more edges from $\cH_2$ than $M$ does (in the graph $B$). So $|m|=2$, all members of $\cH_2$ must belong to $\cM$.
   \end{proof}

To continue the proof of Theorem~\ref{diracH}, let $d:= \delta (G)$, $v\in V$ such that $D:=N_G(v)$, $|D|=d$.
Since $G$ is not hamiltonian, Theorem~\ref{th:D} gives $d\leq k$.
Let $\cH_v$ denote the set of hyperedges of $\cH$ containing the vertex $v$, ($\deg_\cH(v)=|\cH_v|$),
and split it into two parts, $\cH_v= \cD\cup \cL$ where
  $\cD:= \{ e\in E(\cH): v\in e \subseteq \{ v\} \cup D\}$ and $\cL:= \cH_v\setminus \cD$.
Split $\cD$ further into three parts according to the sizes of its edges,
$\cD=\cD^- \cup \cD_2\cup \cD_3$ where $\cD_i:= \{ e\in \cD: |e|=i\}$ (for $i=2,3$) and $\cD^-:= \cD\setminus \left(\cD_2\cup \cD_3\right)$.
Since $\cD$ can have at most $2^d$ members and we handle $\cD_2$ and $\cD_3$ separately we get
\begin{equation}\label{eq333}
|\cD| \leq 2^d-d-\binom{d}{2} +|\cD_2|+ |\cD_3|.
  \end{equation}

Recall that the matching $M$ in the bipartite graph $B$ can be considered as a bijection $\phi: \cM \to \cE$,  where $\cM\subseteq E(\cH)$ and $\cE\subseteq E(H)$.
Define another matching $M_2$ in $B$ by an injection $\phi_2: \cD_2\cup \cD_3\to E(G)$ as follows.
If $m\in \cM \cap (\cD_2\cup \cD_3)$ then $\phi_2(m):= \phi(m)$. In particular, since $\cD_2\subseteq \cM$, if $\{ v,x\} \in \cD_2$ then $\phi_2(\{ v,x\})=  vx$.
If $m= \{ v,x,y\} \in \cD_3 \setminus \cM$ then let $\phi_2(m):= xy$.
Since $\phi_2(\cD_2\cup \cD_3)\subseteq E(G)$ we can apply Theorem~\ref{matroid} to the matchings $M$ and $M_2$ in $B$ with $X_1:= \cM$, $Y_1:= E(G)$, and $X_2:= \cD_2\cup \cD_3$. So there exists a subfamily $\cL_3\subseteq \cH \setminus (\cD_2\cup \cD_3)$ and a bijection
 $\phi_3: (\cD_2\cup \cD_3 \cup \cL_3)\to E(G)$. The matching $M'$ defined by $\phi_3$ is also a largest matching of $B$.
Every $m\in \cL$ has an element $x\notin D$, so $vx\notin E(G)$. If $m$ is not matched in $M'$, then we add $\{m, vx\}$ to $M'$ to get a larger matching. Hence $m \in \cL_3$.
These yield
\begin{equation}\label{eq334}
|\cL|\leq |\cL_3| = e(G)-|\cD_2|-|\cD_3|.  
  \end{equation}
Summing up~\eqref{eq333} and~\eqref{eq334}, then using the lower bound for $|\cH_v|$ and the upper bound
of Theorem~\ref{Erdos} for $e(G)$ we obtain
\begin{equation*}
2^k +1 \leq \deg_\cH(v) \leq 2^d-\binom{d+1}{2} +e(n,d)
  \end{equation*}
The inequality $2^k +1 \leq  2^d-\binom{d+1}{2} +e(n,d)$  does not hold for $n\geq 15$ and $d< k$, e.g., for $(n,k,d) = (16, 7, 6)$, the right hand side is only $64-21+85 = 128$.
This completes the proof of $d=k$.

\subsection{Proof of Theorem~\ref{diracH}, the end}\label{ss34}

We may assume that $G\in \cG_1\cup \dots\cup \cG_5$ by Lemma~\ref{5G},  $\phi$ is a bijection $\phi: \cM \to E(G)$
 with $\phi(m)\subseteq m$ where $\cM\subseteq E(\cH)$, and Claim~\ref{cl33} holds.
Let $\cL_v$ denote the set of edges  $m\in \cH$ containing an edge $vy$ of $E(H)\setminus E(G)$. Note that $\cL_v\subseteq \cM$.
If $\deg_G(v)=k$, then the family $\cL_v$ is non-empty, otherwise $\deg_{\cH}(v) \leq 2^k$. 

Call a graph $F$ with vertex set $V$ a {\em Berge graph of} $\cH$ if $E(F)\subseteq E(H)$, and there exists a subhypergraph $\cF\subseteq \cH$, and a bijection $\psi:\cF \to E(F)$ such that $\psi(m)\subseteq m$ for each $m\in \cF$.
We are looking for a Berge graph of $\cH$ having a hamiltonian cycle.
In particular, the graph $G$ is a Berge graph of $\cH$ and it is almost hamiltonian. We will show that a slight change to $G$ yields a hamiltonian Berge graph of $\cH$.

If $G\in \cG_2\cup \cG_3$ then choose any $v\in V_1\setminus \{ x_0\}$  and let $m\in \cL_v$. There exists an edge $vy\in \left(E(H)\setminus E(G)\right)$ contained in $m$. Then $y\in V_2\setminus \{ x_0\}$. 
The graph $\left(E(G)\setminus \{ \phi(m) \}\right)\cup \{ vy\} $ is a Berge graph of $\cH$ (we map $m$ to the edge $vy$ instead of $\phi(m)$).
According to Lemma~\ref{5Gmodified} (with $\cA:= \{ \phi(m)\}$ and  $\cB:= \{ vy\}$) it is hamiltonian except if we run into the only exceptional case:
$x_0$ has exactly two $G$-neighbors $x_2$ and $y_2$ in $V_2$,  $vy=vx_2$, and $\phi(m)=x_0y_2$.
In this case $m$ contains $\{ x_0,v, x_2, y_2\}$ so it can be avoided by choosing $y:= y_2$ instead of $y=x_2$.

If $G\in \cG_4$ then we argue in a very similar way.
Choose any $v\in V_2$  and let $m\in \cL_v$ containing an edge $vy\in \left(E(H)\setminus E(G)\right)$. Then $y\in V_2$ and
 the graph $\left(E(G)\setminus \{ \phi(m) \}\right)\cup \{ vy\} $ is a Berge graph of $\cH$ that is hamiltonian by Lemma~\ref{5Gmodified} with $\cA:= \{ \phi(m)\}$ and  $\cB:= \{ vy\}$.
From now on we may suppose that $n=2k+2$ so $|\cL_v|\geq 2$ for $\deg_G(v)=k$.

If $G\in \cG_1$ then define $\cM_{1,2}$ as the members of $\cM$ meeting both $V_1$ and $V_2$.
The minimum degree condition on $\cH$ implies that $|\cM_{1,2}|\geq 2$.
Since $\cM_{1,2}$ can have at most one member of size $2$, we can choose an $m_1$, $|m_1|\geq 3$. By symmetry we may suppose that $|m_1\cap V_1|\geq 2$ and let $x_2\in V_2\cap m_1$.
Choose an element $y\in V_2$, $y\notin e_0$, $y\neq x_2$.
Since $|\cL_y|\geq 2$ we can choose an $m_2\in \cM_{1,2}$ such that $m_1\neq m_2$ and $y\in m_2$.
Take any pair $\{ y_1,y\}\subseteq m_2$ with $y_1\in V_1$.  Then one can choose an $x_1\in m_1\cap V_1$ so that $x_1\neq x_2$. So the pairs $\{ x_1, x_2\}\subseteq m_1$ and $\{ y_1, y\}\subseteq m_2$ are disjoint.
Lemma~\ref{5Gmodified} with $\cA:= \{ \phi(m_1), \phi(m_2)\}$ and $\cB:= \{ x_1x_2, y_1y\}$ implies that the graph
$\left(E(G)\setminus \cA \right)\cup \cB $ is a hamiltonian Berge graph of $\cH$.

If $G\in \cG_5$ then $|\cL_v|\geq 2$ for any $v\in V_2\setminus e_0$ and for all members $m$ of $\cL_v$ we have $|m\cap V_2|\geq 2$.
Fix $v\in V_2\setminus e_0$ and let $m_1$ be an arbitrary member of $\cL_v$. Choose a pair $\{ v,v'\} \subseteq m_1\cap V_2$.
Fix another vertex $u\in  V_2\setminus (e_0\cup \{ v,v'\})$ and let $m_2$ be an arbitrary member of $\cL_u$. Choose a pair $\{ u,u'\} \subseteq m_2\cap V_2$. Then $u\notin \{ v,v'\}$ so the pairs  $\{ u,u'\}$ and $\{ v,v'\} $ are distinct. Again, apply Lemma~\ref{5Gmodified} with $\cA:= \{ \phi(m_1), \phi(m_2)\}$ and $\cB:= \{ uu', vv'\}$. This completes the proof of Theorem~\ref{diracH}.
\qed

\begin{rem}
We can also show that all extremal examples are slight modifications of the four types of the sharpness examples described after Theorem~\ref{diracH}.
   \end{rem}

\section{Dirac-type conditions for long Berge cycles}\label{secCk}

In this section we prove Theorem~\ref{th:mindeg_path} for Berge paths and Theorem~\ref{th:mindeg_cycle} for Berge cycles. In fact we prove the two statements simultaneously.
\begin{proof}[Proof of Theorems~\ref{th:mindeg_path} and~\ref{th:mindeg_cycle}]
Suppose that $\delta(\cH)\geq 2^{k-2}+1$, $k\geq 3$ and that $\cH$ has no Berge cycle of length $k$ or longer.
We will show that it contains a Berge path of length $k-1$ (thus establishing Theorem~\ref{th:mindeg_path}) and then that $\delta(\cH)= 2^{k-2}+1$ (which completes the proof of Theorem~\ref{th:mindeg_cycle}).

Choose a longest Berge path in $\cH$ according the following rules.
We say that a Berge path 
 with edges $\{e_1,\dots,e_s\}$ is \emph{better} than a Berge path 
 with edges $\{f_1,\dots,f_t\}$ if
\newline\indent${}$\quad
a) $s>t$  or
\newline\indent${}$\quad
b) $s=t$ and $\sum |e_i|<\sum |f_j|$.

Consider a best Berge path $\cP$ in $\cH$.
Let the base vertices of the path be $v_1, v_2, \dots,  v_p$.
Let $e_1, \dots ,e_{p-1}$ be the edges of the path ($v_{i}, v_{i+1}\in e_i$).
First, we show that $p\geq k-1$.
(In fact, $p\geq k$ follows but that will be proved later).

Indeed, let $\cH^{(p)}$ be the hypergraph consisting of the edges of $\cH$
 containing $v_p$, contained in $\{ v_1, \dots, v_p \}$ and also the edges of the path, i.e.,
\[
 E(\cH^{(p)}):=\{ e\in E(\cH): v_p\in e\subseteq \{ v_1, \dots, v_p\}\} \cup \{e_1, \dots ,e_{p-1} \}.
 \]
Then for $p\leq k-2$ (and $k\geq 3$) we have
\[
|E(\cH^{(p)})|\leq 2^{p-1}+ (p-1)\leq 2^{k-2}< \delta(\cH)\leq \deg_\cH(v_p).
\]
So there exits an edge $f$ in $E(\cH)\setminus E(\cH^{(p)})$ containing $v_p$.
Then  $e_1, \dots ,e_{p-1},f$ form a Berge path longer than $\cP$, a contradiction.

Now we have $p\geq k-1$, so we can define $W:=\{ v_1, \dots, v_{k-1}\}$. Let $\cP_1$ be the subhypergraph consisting of the first $k-1$  edges of $\cP$,
  $E(\cP_1):= \{ e_1, ..., e_{k-1}\}$
  (if $p=k-1$ we take $\cP_1:=\cP$).
Let $\cH_1$ be the subhypergraph of $\cH$ consisting of the edges incident to $v_1$.

\begin{claim}\label{cl:52}
Every edge $f\in E(\cH_1)\setminus E(\cP_1)$ is contained in $W:=\{ v_1, \dots, v_{k-1}\}$.
\end{claim}
\begin{proof}
First, we show that every edge $f\in E(\cH_1)\setminus E(\cP_1)$ avoids $\{v_k, \dots, v_p\}$.
Otherwise, if there exists an edge $f\in E(\cH_1)\setminus E(\cP_1)$ such that
  $f\cap \{v_k, \dots, v_p\} \neq \emptyset$, then suppose that
  $v_i$ has the minimum index ($k\leq i\leq p $) such that $v_i$ is a vertex of such an $f$.
Then $e_1, \dots, e_{i-1}$ and $f$ are forming a Berge cycle of length $i$, since these hyperedges are all distinct and $v_1,v_i\in f$.
Finally, suppose that there is an edge  $f\in E(\cH_1)\setminus E(\cP_1)$
  such that $v\in f$, $v\notin W$.
Then $v\notin \{ v_1, \dots, v_p\}$ so the path $f, e_1, \dots, v_p$ is longer than $\cP$, a contradiction.
\end{proof}

Let $\cK$ be the family of all $2^{k-2}$ subsets of $W$ that  contain $v_1$. We claim there is a one-to-one mapping $\varphi$ from $\cH_1 \setminus e_{k-1}$ to $\cK$. The existence of such a $\varphi$ implies
\begin{equation}\label{eq41}
  \delta(\cH)\leq \deg_\cH(v_1) \leq 2^{k-2} + 1.
\end{equation}

If an edge $e$ of $\cH_1$ satisfies $e \subseteq W$, then let $\varphi(e) = e$. Otherwise, let $\cA \subseteq \cH_1$  be the set of the edges of $\cH \setminus \{e_{k-1}\}$ that contain both $v_1$ and some vertex outside of $W$. By Claim~\ref{cl:52}, each $e \in \cA$ must be some edge $e_i$ in $\cP_1$. Hence it remains to show that all elements of $\cA$ can be mapped to  distinct elements of $\cK$ that are not  edges of $\cH$.

Observe that if $e_i\in \cA$ then $\{v_i,v_{i+1}\}\notin \cH$.
Otherwise, we get a better path by replacing $e_i$ by $\{v_i,v_{i+1}\}$.
Also, for $1\leq i\leq k-2$, $e_i\in \cA$ implies $v_1\in e_i$ and $\{v_i,v_{i+1}\}\subset e_i$.
Since $e_i\not \subset W$ we get $|e_i|\geq 4$ for $i\geq 2$.
We also obtain that in case of $i\geq 3$, $e_i\in \cA$ we have $\{v_1, v_i,v_{i+1}\}\notin \cP $, and moreover $\{v_1,v_i,v_{i+1}\} \notin \cH$ since otherwise we get a better path by replacing $e_i$ by $\{v_1, v_i,v_{i+1}\}$.
For $3\leq i\leq k-2$ (and $e_i\in \cA$) define $\varphi(e_i)$ as $\{v_1, v_i,v_{i+1}\}$.

If $e_2\in \cA$  and $\{ v_1, v_2, v_3 \}\not \in \cH$ then
 we proceed as above,  $\varphi(e_2):=\{ v_1, v_2, v_3 \}$. Otherwise, if $e_2\in \cA$ (so $|e_2|\geq 4$) and $\{ v_1, v_2, v_3 \}\in \cH$ then
 $\{ v_1, v_2, v_3 \}\in \cP$ too (otherwise, we get a better path by replacing $e_2$ by $\{v_1, v_2,v_3\}$).
We get $e_1=\{ v_1, v_2, v_3 \}$ (and $e_1 \subset e_2$).
We claim that $\{ v_1, v_3\}\notin \cH$.
Otherwise we rearrange the base vertices of the path $\cP$ by exchanging $v_1$ and $v_2$ (and get the order $v_2, v_1, v_3, \dots, v_{p}$) and observe that the Berge path  $\{ v_2, v_1, v_3\}, \{ v_1, v_3 \}, e_3, \dots, e_{p-1}$ is better than $\cP$, a contradiction. So in this case $\varphi(e_2):= \{ v_1, v_3 \}$. Finally, if $e_1\in \cA$ then $\varphi(e_1):= \{ v_1, v_2 \}$, and the definition of $\varphi$ is complete.

We have shown that $\deg_\cH(v_1) \leq |\cH_1 \setminus \{e_{k-1}\}| +1 \leq 2^{k-2} +1$. Equality holds, so $v_1 \in e_{k-1}$. In particular $e_{k-1}$ must exist, so $\cP$ was a Berge path of length at least $k-1$.
\end{proof}

\medskip
Our method works for multihypergraphs as well.
If the maximum multiplicity of an edge is $\mu$, then the corresponding necessary bounds on the minimum degrees are $\mu 2^{k-2}+1$ or $\mu 2^{k-2}+2$, respectively.
Indeed, suppose that  $\delta(\cF)\geq \mu 2^{k-2}+1$, $k\geq 3$ and that $\cF$ has no Berge cycle of length $k$ or longer.
Let $\cH$ be the simple hypergraph obtained from $\cF$ by keeping one copy from the multiple edges. We have $\delta(\cH)\geq 2^{k-2}+1$.
Then Theorems~\ref{th:mindeg_path} implies that $\cH$ (and $\cF$ as well) contain a Berge path with $k$ base vertices.

As in the proof of Theorem~\ref{th:mindeg_cycle}, consider a best Berge path $\cP$ in $\cH$ with
 base vertices  $v_1, v_2, \dots,  v_p$ and edges $e_1, \dots ,e_{p-1}$. We have $p\geq k$.
Then~\eqref{eq41} gives $\deg_\cH(v_1)= 2^{k-2} + 1$ and we get
 $\deg_\cH(v_1)= |\cH_1 \setminus \{e_{k-1}\}| +1$.
Since we also obtained $\{v_1, v_{k-1}, v_k\} \subset e_{k-1}$, the multiplicity of $e_{k-1}$ could not exceed $1$. So
$\delta(\cF)$ could not exceed $\mu 2^{k-2}+1$.

\section{Maximum number of edges} 
\noindent{\em Proof of Theorem~\ref{th:EGh}.}\enskip
Suppose that among all $n$-vertex hypergraphs with $c(\cH)<k$ and $e(\cH)$ edges our $\cH$ is chosen so that $\sum_{e \in E(\cH)} |e|$ is minimized.

We claim that $\cH$ is a downset, that is, for any $e \in E(\cH)$ and $e' \subset e$, $e' \in E(\cH)$. 
Indeed, if there exists a set $e'$ and a hypergedge $e$ such that  $e' \subset e$ such that $e' \notin E(\cH)$ and $e \in E(\cH)$, then the hypergraph obtained by replacing $e$ with $e'$ also does not contain a Berge cyle of length $k$ or longer. This contradicts the choice of $\cH$.

Let $H = \partial_2\cH$ be the 2-shadow of $\cH$. Suppose that $H$ contains a cycle $C = v_1 v_2 \ldots v_\ell v_1$. Every edge $v_iv_{i+1}$ of $C$ is contained in a hyperedge of $\cH$. But since $\cH$ is a downset, the hyperedge $\{v_i, v_{i+1}\}$ is also contained in $E(\cH)$. Therefore $\cH$ also contains a (Berge) cycle of length $\ell$. Hence the graph $H$ contains no cycles of length at least $k$.

Let $e_r(\cH)$ be the number of hyperedges of $\cH$ of size $r$. In $H$, every hyperedge $e$ of $\cH$ is represented by a clique of order $|e|$, and so $e_r(\cH)$ is at most the number of cliques of size $r$ in $H$.
Since $c(H)<k$, each hyperedge contains at most $k-1$ vertices.
By Theorem~\ref{cliques},
\[e(\cH) = e_0(\cH)+ e_1(\cH)+ \sum_{r=2}^{k-1} e_r(\cH) \leq 1+ n+ \sum_{r=2}^{k-1}\frac{n-1}{k-2}{k-1 \choose r} = 2+ \frac{n-1}{k-2}\left(2^{k-1}-2\right). \qed \]

\end{document}